\documentclass[11pt]{article}

\usepackage{amsmath,amsfonts,amssymb,color,diagrams}

\newcommand{\beq}{\begin{equation}}
\newcommand{\eeq}{\end{equation}}
\newcommand{\bea}{\begin{eqnarray}}
\newcommand{\eea}{\end{eqnarray}}
\newcommand{\beas}{\begin{eqnarray*}}
\newcommand{\eeas}{\end{eqnarray*}}

\definecolor{dg}{rgb}{0, 0.5, 0}

\newtheorem{theorem}{Theorem}[section]

\newtheorem{definition}[theorem]{Definition}
\newtheorem{proposition}[theorem]{Proposition}

\newtheorem{prop}[theorem]{Proposition}

\newtheorem{lemma}[theorem]{Lemma}
\newtheorem{remark}[theorem]{Remark}
\newtheorem{example}[theorem]{Example}
\newtheorem{examples}[theorem]{Examples}
\newtheorem{foo}[theorem]{Remarks}

\newenvironment{proof}{\addvspace{\medskipamount}\par\noindent{\it
Proof}.}
{\unskip\nobreak\hfill$\Box$\par\addvspace{\medskipamount}}

\newcommand{\bS}{\mathbb S}
\newcommand{\BS}{\bS^{4n+3}}

\newcommand{\R}{\mathbb R}

\newcommand{\cp}{\mathbb {CP}}

\title{The Subelliptic Heat Kernels of the Quaternionic Hopf Fibration}
\author{Fabrice Baudoin, Jing Wang }

\date{Department of Mathematics, Purdue University \\
 West Lafayette, IN, USA}

\begin{document}
\maketitle

\begin{abstract}
The main goal of this work is to study the sub-Laplacian of the unit sphere  which is obtained by lifting with respect to the Hopf fibration the Laplacian of the quaternionic projective space. We obtain in particular explicit formulas for its heat kernel and deduce an   expression for the Green function of the conformal sub-Laplacian and small-time asymptotics. As a byproduct of our study we also obtain several results related to the sub-Laplacian of a projected Hopf fibration.
\end{abstract}
\tableofcontents
\clearpage
\section{Introduction}

The first objective of this work is to study and give meaningful formulas for the heat kernel of the sub-Laplacian which is associated to the quaternionic Hopf fibration on $\mathbb{S}^{4n+3}$:
\[
\mathbf{SU}(2)\to\bS^{4n+3}\to\mathbb{HP}^n.
\] 
 This fibration originates from the natural action of  $\mathbf{SU}(2)$ on $\mathbb{S}^{4n+3}$, which identifies  $\mathbb{S}^{4n+3}$ as a $\mathbf{SU}(2)$ bundle over the projective quaternionic space $ \mathbb{HP}^n$. The sub-Laplacian we are interested in appears then as the lift on $\mathbb{S}^{4n+3}$ of the Laplace-Beltrami operator of $ \mathbb{HP}^n$. As we will see, it is intrinsically associated to the quaternionic contact geometry of $\mathbb{S}^{4n+3}$. 
 
 \
 
  Let us briefly describe our main results. One of the first observations is that, due to the symmetries of the above fibration, the heat kernel of the sub-Laplacian only depends on two variables: the variable $r$ which is the Riemannian distance on $\mathbb{HP}^n$ and the variable $\eta$ which is the Riemannian distance on the fiber $\mathbf{SU}(2)$.  We prove that in these coordinates, the cylindric part of the sub-Laplacian writes
 \[
\frac{\partial^2}{\partial r^2}+((4n-1)\cot r-3\tan r)\frac{\partial}{\partial r}+\tan^2r \left(\frac{\partial^2}{\partial \eta^2}+2\cot \eta\frac{\partial}{\partial \eta}\right).
\]
As a consequence of this expression for the sub-Laplacian, we are able to derive two expressions for the heat kernel: 
\begin{itemize}
\item[(1)]  A Minakshisundaram-Pleijel spectral expansion:
\begin{equation*}
p_t(r,\eta)=\sum_{m=0}^{+\infty}\sum_{k=0}^{\infty}\alpha_{k,m}e^{-4[k(k+2n+m+1)+nm]t}\frac{\sin (m+1)\eta}{\sin \eta}(\cos r)^mP_k^{2n-1,m+1}(\cos 2r)
\end{equation*}
where $\alpha_{k,m}=\frac{\Gamma(2n)}{2\pi^{2n+2}}(2k+m+2n+1)(m+1){k+m+2n\choose 2n-1}$ and $P_k^{2n-1,m+1}$ is a Jacobi polynomial. This formula is  useful  to study the long-time behavior of the heat kernel but seems difficult to use in the study of small-time asymptotics or for the purpose of proving upper and lower bounds. In order to derive small-time asymptotics of the kernel, we give another analytic expression for $p_t(r, \theta)$ which is much more geometrically meaningful. 
\item[(2)] An integral representation:
\[
p_t(r, \eta)=\frac{e^{-t}}{\sqrt{\pi t}} \int_0^{+\infty} \frac{ \sinh y \sin \left(  \frac{\eta y}{2t}\right) }{\sin \eta} e^{-\frac{y^2-\eta^2}{4t}} q_t( \cos r\cosh y ) dy
\]
where $q_t$ is the heat kernel of the Riemannian structure of $\bS^{4n+3}$. We obtain this formula by employing a similar idea that was developed in the usual Hopf fibrations  (see \cite{BB}, \cite{CRS}, \cite{Bo}, \cite{W}). The key point is the commutation between the sub-Laplacian and the transversal directions in $\mathbf{SU}(2)$.
From this formula we are able to deduce the fundamental solution of the conformal sub-Laplacian $-L+4n(n+1)$. Furthermore, we also derive three different behaviors of the small-time asymptotics of the heat kernel: on the diagonal, on the vertical cut-locus, and outside of cut-locus. An interesting by-product of this small-time asymptotics we obtain, is a previously unknown explicit formula for the sub-Riemannian distance on the quaternionic unit sphere.
\end{itemize}

 \
 
 The second main objective of this work is the study of another sub-Laplacian that we now define.  The natural action of $\mathbb{S}^1$ on $\mathbb{S}^{4n+3}$ induces the classical Hopf fibration 
 \[
  \mathbb{S}^1 \to \bS^{4n+3} \to  \mathbb{CP}^{2n+1}
 \]
whose sub-Laplacian and corresponding heat kernel were studied in details in our previous work \cite{CRS}. Identifying $\mathbb{S}^1$ with a subgroup of $\mathbf{SU}(2)$, defines  a fibration
\[
\mathbf{SU}(2)  / \mathbb{S}^1 =  \mathbb{CP}^1 \to \mathbb{CP}^{2n+1} \to \mathbb{HP}^n
\]
that makes the following diagram commutative
\begin{diagram}\label{diag}
  & & \mathbb{S}^1 & & \\
  & \ldTo & \dTo & & \\
 \mathbf{SU}(2) & \rTo &\mathbb{S}^{4n+3} & \rTo & \mathbb{HP}^n \\
 \dTo & &\dTo & \ruTo & \\
 \mathbb{CP}^1 & \rTo & \mathbb{CP}^{2n+1} & & 
\end{diagram}
We are interested in the sub-Laplacian of the complex projective space $\mathbb{CP}^{2n+1}$ which is obtained by lifting the Laplace-Beltrami operator of $\mathbb{HP}^n$. Again, our main goal will be to provide meaningful formulas for the heat kernel of this sub-Laplacian. For the very same reasons as above,  the heat kernel $h_t$ of this sub-Laplacian only depends on two variables: $r$ which is again the Riemannian distance on $\mathbb{HP}^n$ and $\phi$ which is the Riemannian distance on $ \mathbb{CP}^1$. The expression of the cylindric part of the sub-Laplacian is then given by:
\[
\frac{\partial^2}{\partial r^2}+((4n-1)\cot r-3\tan r)\frac{\partial}{\partial r}+\tan^2r\left( \frac{\partial^2}{\partial \phi^2}+2\cot 2\phi\frac{\partial}{\partial \phi}\right)
\]
The Minakshisundaram-Pleijel  expansion of $h_t$  can then  be deduced in the same fashion as on $\bS^{4n+3}$, and by comparing it with the previous spectral decomposition of $p_t$, we prove the  following intertwining between the two kernels:
\[
h_t(r,\cos2\phi)=\frac{1}{2\pi}\int_0^\pi p_t(r,\cos\phi\cos\theta)d\theta.
\]
As a consequence we obtain in particular the small-time asymptotics of $h_t$.

\

To put our work in perspective, we mention that the study of subelliptic heat kernels on model spaces has generated quite a lot of interest in the past and is still nowadays an active domain of research which lies at the intersection of harmonic analysis, partial differential equations, control theory, differential geometry and probability theory (see the book \cite{CCFI} for an overview). One of the first studies goes back to Gaveau \cite{Ga} who provided an expression for the subelliptic  heat kernel on the simplest sub-Riemannian model space, the 3-dimensional Heisenberg group. The subelliptic heat kernel of the 3-dimensional Hopf fibration on $\mathbb{S}^3$ was first studied by Bauer \cite{bauer} and then, in more details by Baudoin and Bonnefont \cite{BB}. The subelliptic heat kernel of the 3-dimensional Hopf fibration on $\mathbf{SL}(2)$ was then studied by Bonnefont \cite{Bo}. A general study of heat kernels on any 3-dimensional contact manifold is then presented in \cite{BC}. The $n$-dimensional generalization of the work by Gaveau \cite{Ga} was made by Beals, Gaveau and Greiner in \cite{BGG}. The $n$-dimensional generalization of the work by Baudoin-Bonnefont \cite{BB} was made by the two present authors in \cite{CRS}. We also mention the work \cite{MM} and the work by Greiner \cite{Gr} who recovers results of \cite{CRS} by using the Hamiltonian method. The $n$-dimensional generalization of the work by Bonnefont \cite{Bo} was made by the second author of the present paper in \cite{W}. We also point out the work by Agrachev-Boscain and Gauthier \cite{A} that studies a general class of subelliptic heat kernels on Lie groups.

\

To conclude, we can observe that, up to an exotic example,  due to the work of Escobales \cite{Esc} the submersions $\mathbb{S}^{2n+1} \to \mathbb{CP}^n$ and $\mathbb{S}^{4n+3} \to \mathbb{HP}^n$ are  the only examples of Riemannian submersions of the sphere with totally geodesic fibers. As a consequence our work is a perfect complement of \cite{CRS} and completes hence  the study of all the natural subelliptic heat kernels of the unit sphere that come from a Riemannian submersion.

\

\section{The subelliptic heat kernel on $\mathbb{S}^{4n+3}$ }

The sub-Riemannian geometry of $\mathbb{S}^{4n+3}$ we are interested in, may be defined in at least two ways. The first one is to consider the quaternionic Hopf fibration
\[
\mathbf{SU}(2)\to\bS^{4n+3}\to\mathbb{HP}^n.
\]
Lifting the Laplace Beltrami operator of $\mathbb{HP}^n$ with respect to this submersion gives the sub-Laplacian of $\mathbb{S}^{4n+3}$ we want to study. This point of view gives a quick way to prove Proposition \ref{ger} below. 

The second way to study the sub-Riemannian geometry of $\mathbb{S}^{4n+3}$ is to see $\mathbb{S}^{4n+3}$ as a quaternionic contact manifold (more precisely a 3-Sasakian manifold). The sub-Laplacian on $\mathbb{S}^{4n+3}$ is then simply defined as the trace of the horizontal Hessian for the Biquard connection. For the sake of completeness we give a proof of  Proposition \ref{ger}  by using the quaternionic contact point of view, because it is natural to really see $\mathbb{S}^{4n+3}$ as the model space of a positively curved quaternionic contact manifold. 

The reader more interested in the analysis of the sub-Laplacian than in the geometry associated to it may skip Section 2.1. and admit Proposition \ref{ger}. Once Proposition \ref{ger} is admitted the remainder of the paper may be read without further references to quaternionic contact geometry.

\subsection{The quaternionic contact structures and the unit spheres}

Quaternionic spheres appear as the model spaces of quaternionic contact manifolds and 3-Sasakian manifolds (see \cite{BG,IMV,IPV}) . We introduce them as follows: the quaternionic unit sphere is given by
\[
\bS^{4n+3}=\lbrace q=(q_1,\cdots,q_{n+1})\in \mathbb{H}^{n+1}, \| q \| =1\rbrace,
\]
where we denote  the quaternionic field by
\[
\mathbb{H} =\{ q=t+x I +y J +z K, (t,x,y,z) \in \mathbb{R}^4 \},
\]
where $I,J,K$ are the Pauli matrices:
\[
I=\left(
\begin{array}{ll}
i & 0 \\
0&-i 
\end{array}
\right), \quad 
J= \left(
\begin{array}{ll}
0 & 1 \\
-1 &0 
\end{array}
\right), \quad 
K= \left(
\begin{array}{ll}
0 & i \\
i &0 
\end{array}
\right).
\]
The quaternionic norm is 
\[
\| q \|^2 =t^2 +x^2+y^2+z^2.
\]

The sub-Riemannian structure of the quaternionic spheres comes from the quaternionic Hopf fibration:
\[
\mathbf{SU}(2)\to\bS^{4n+3}\to\mathbb{HP}^n,
\] 
 that we now describe. There is natural and isometric group action of the Lie group $\mathbf{SU}(2)$ on $\bS^{4n+3}$ which is given by, 
 $$ g \cdot (q_1,\cdots, q_{n+1}) = (g q_1,\cdots, g q_{n+1}). $$
 
 For any $f\in C^\infty(\bS^{4n+3})$, the infinitesimal generator of the left translation by $e^{I\theta}$ is given by
 \[
 \frac{d}{d\theta}f(e^{I\theta}q)\mid_{\theta=0}=\sum_{i=1}^{n+1}\frac{\partial f}{\partial t_i}\frac{Iq_i-\overline{q_i}I}{2}+\frac{\partial f}{\partial x_i}\frac{q_i+\overline{q_i}}{2}+\frac{\partial f}{\partial y_i}\frac{Kq_i-\overline{q_i}K}{2}+\frac{\partial f}{\partial z_i}\frac{\overline{q_i}J-Jq_i}{2},
 \]
 that is, 
 \[
 \frac{d}{d\theta}f(e^{I\theta}q)\mid_{\theta=0}=\sum_{i=1}^{n+1}\left(-x_i\frac{\partial f}{\partial t_i}+t_i\frac{\partial f}{\partial x_i}-z_i\frac{\partial f}{\partial y_i}+y_i\frac{\partial f}{\partial z_i}\right),
 \]
 where we use the real coordinates in $\mathbb{R}^{n+1}$, $(t_1, x_1,y_1,z_1, \cdots, t_{n+1},x_{n+1},y_{n+1},z_{n+1})$.
 Similarly, since
 \[
  \frac{d}{d\theta}f(e^{J\theta}q)\mid_{\theta=0}=\sum_{i=1}^{n+1}\frac{\partial f}{\partial t_i}\frac{Jq_i-\overline{q_i}J}{2}+\frac{\partial f}{\partial x_i}\frac{\overline{q_i}K-Kq_i}{2}+\frac{\partial f}{\partial y_i}\frac{q_i+\overline{q_i}}{2}+\frac{\partial f}{\partial z_i}\frac{Iq_i-\overline{q_i}I}{2},
 \]
 and
 \[
  \frac{d}{d\theta}f(e^{K\theta}q)\mid_{\theta=0}=\sum_{i=1}^{n+1}\frac{\partial f}{\partial t_i}\frac{Kq_i-\overline{q_i}K}{2}+\frac{\partial f}{\partial x_i}\frac{Jq_i-\overline{q_i}J}{2}+\frac{\partial f}{\partial y_i}\frac{\overline{q_i}I-Iq_i}{2}+\frac{\partial f}{\partial z_i}\frac{\overline{q_i}+q_i}{2}.
 \]
 we have that 
 \[
 \frac{d}{d\theta}f(e^{J\theta}q)\mid_{\theta=0}=\sum_{i=1}^{n+1}\left(-y_i\frac{\partial f}{\partial t_i}+z_i\frac{\partial f}{\partial x_i}+t_i\frac{\partial f}{\partial y_i}-x_i\frac{\partial f}{\partial z_i}\right).
 \]
 and
 \[
 \frac{d}{d\theta}f(e^{K\theta}q)\mid_{\theta=0}=\sum_{i=1}^{n+1}\left(-z_i\frac{\partial f}{\partial t_i}-y_i\frac{\partial f}{\partial x_i}+x_i\frac{\partial f}{\partial y_i}+t_i\frac{\partial f}{\partial z_i}\right).
 \]
 In quaternionic coordinates since  $\frac{\partial}{\partial{q_i}}= \frac{1}{2}\left(\frac{\partial}{\partial{t_i}}-I \frac{\partial}{\partial{x_i}}-J \frac{\partial}{\partial{y_i}}-K \frac{\partial}{\partial{z_i}}\right)$, let $T=\sum_{i=1}^{n+1}q_i \frac{\partial}{\partial{q_i}}-\frac{\partial}{\partial{\overline{q_i}}}\overline{q_i}$ and denote $T=IT_1+JT_2+KT_3$, where $T_1,T_2, T_3$ are real vector fields on $T(\bS^{4n+3})$. We then have that
 \[
 \frac{d}{d\theta}f(e^{I\theta}q)\mid_{\theta=0}=-(IT+TI)=2T_1,
 \]
  \[
 \frac{d}{d\theta}f(e^{J\theta}q)\mid_{\theta=0}=-(JT+TJ)=2T_2,
 \]
 and
  \[
 \frac{d}{d\theta}f(e^{K\theta}q)\mid_{\theta=0}=-(KT+TK)=2T_3.
 \]
They form a basis of the fibers of a $\mathbf{SU}(2)$- bundle structure on $\bS^{4n+3}$. 

We denote the one form  $dq_i=dt_i+Idx_i+Jdy_i+Kdz_i$. Simple computations show that for all $1\leq i,j\leq n+1$, $S=I,J,K$,
\[
dq_i(\frac{\partial}{\partial q_j})=d\overline{q_i}(\frac{\partial}{\partial \overline{q_j}})=2\delta_{ij}, \  dq_i(S\frac{\partial}{\partial q_j})=d\overline{q_i}(S\frac{\partial}{\partial \overline{q_j}})=0,
\]
and
\[
dq_i(\frac{\partial}{\partial \overline{q_j}})=d\overline{q_i}(\frac{\partial}{\partial {q_j}})=-\delta_{ij},\ 
dq_i(S\frac{\partial}{\partial \overline{q_j}})=d\overline{q_i}(S\frac{\partial}{\partial {q_j}})=S\delta_{ij}.
\]
Therefore, we have that for all $p\in\mathbb{H}$,
\[
dq_i(p\frac{\partial}{\partial q_j})=d\overline{q_i}(p\frac{\partial}{\partial \overline{q_j}})=\delta_{ij}(p+\overline{p}),\ dq_i(p\frac{\partial}{\partial \overline{q_j}})=d\overline{q_i}(p\frac{\partial}{\partial q_j})=-\delta_{ij}\overline{p}
\]

We choose the contact form $\eta=\frac{1}{2}\sum_{q=1}^{n+1}(dq_i)\overline{q_i}-q_id\overline{q_i}$. It has only imaginary part and we denote by $\eta=\eta_1I+\eta_2J+\eta_3K$ where $(\eta_1,\eta_2,\eta_3)$ are real contact forms. 

Let $T=\sum_{i=1}^{n+1}q_i \frac{\partial}{\partial{q_i}}-\frac{\partial}{\partial{\overline{q_i}}}\overline{q_i}$, then $\eta(T)=3$ and clearly $T=-\overline{T}$. If  we assume that $T=IT_1+JT_2+KT_3$, then
\[
T_1=-\frac{1}{2}(IT+TI),\ T_2=-\frac{1}{2}(JT+TJ),\ T_3=-\frac{1}{2}(KT+TK).
\] 
We can easily compute that 
\[
\eta(T_1)=-I,\ \eta(T_2)=-J,\ \eta(T_3)=-K,
\]
and
\[
\eta(ST)=-S, \quad S=I,J,K.
\]
Therefore, the real contact forms $(\eta_1,\eta_2,\eta_3)$ given by
\[
\eta_1=\frac{1}{2}(I\eta+\eta I),\ \eta_2=\frac{1}{2}(J\eta+\eta J),\ \eta_3=\frac{1}{2}(K\eta+\eta K),
\]
satisfy that
\[
\eta_i(T_j)=\delta_{ij}.
\]
The horizontal distribution $\mathbf{H}$ of $\bS^{4n+3}$ is given by the kernel of $\eta$ and $d\eta$ induces the quaternionic Hermitian structure $g_\eta$ on $\mathbf{H}$: for all $X,Y\in \mathbf{H}$,
\[
g_\eta(X, I_1Y)=\frac{1}{2}d\eta_1(X,Y),\ g_\eta(X, I_2Y)=\frac{1}{2}d\eta_2(X,Y),\ g_\eta(X, I_3Y)=\frac{1}{2}d\eta_3(X,Y).
\]
It is compatible with the Hermitian structure in the sense that
\[
g_\eta(I_k\cdot,I_k\cdot)=g_\eta(\cdot,\cdot), \quad k=1,2,3.
\]
Notice that $T(\bS^{4n+3})=\mathbf{H}\oplus \mathbf{span}\{T_1,T_2,T_3\}$, and we can obtain the semi-Riemannian metric $g$ on $\bS^{4n+3}$ by extending $g_\eta$ as follows:
\[
g(X,Y)=g_\eta(X,Y),\ g(X,T_i)=0,\ g(T_i, T_j)=\delta_{ij}
\]
for all $X,Y\in\mathbf{H}, 1\leq i,j\leq3$.

\subsection{The sub-Laplacian on $\bS^{4n+3}$}
On a general contact quaternionic manifold whose vertical space is generated by three Reeb vector fields, there exists a canonical connection $\nabla$ which preserves the metric and the almost complex structure (see \cite{biquard}). It is called the Biquard connection. 
We now introduce the canonical sub-Laplacian on $\bS^{4n+3}$ as follows: for any $f\in C^2(\bS^{4n+3})$,
\begin{equation}\label{subL}
Lf=\mathbf{trace}_{g_\eta}(\pi_{\mathbf{H}}\nabla^2f)
\end{equation}
where $\nabla^2$ is the pseudo-Hermitian Hessian of $f$ with respect to the Biquard connection, and $\pi_{\mathbf{H}}B$ is the restriction of $B$ to $\mathbf{H}(\bS^{4n+3})\otimes\mathbf{H}(\BS)$ for any bilinear form $B$ on $T(\bS^{4n+3})$.\

To compute $L$ in local coordinates, we project the vector fields $\frac{\partial}{\partial t_i}$ onto $\mathbf{H}(\BS)$ and obtain  the horizontal vector fields: 
since $\frac{\partial}{\partial t_i}=\frac{\partial}{\partial q_i}+\frac{\partial}{\partial \overline{q_i}}$, and
\[
\eta(\frac{\partial}{\partial q_i})=\frac{2q_i+\overline{q_i}}{2}, \ \eta(\frac{\partial}{\partial \overline{q_i}})=-\frac{3}{2}q_i.
\]
\[
\eta({\frac{\partial}{\partial t_i}})=\frac{\overline{q_i}-q_i}{2}=-(x_iI+y_iJ+z_iK).
\]
Therefore
\[
X_{0,i}=\frac{\partial}{\partial t}+\frac{(x_iI+y_iJ+z_iK)}{3}T,\quad \mbox{for all }1\leq i\leq n+1.
\]
Similarly, we can obtain $X_{1,i}$, $X_{2,i}$, $X_{3,i}$ by projecting $\frac{\partial}{\partial x_i}$, $\frac{\partial}{\partial y_i}$, and $\frac{\partial}{\partial z_i}$ as follows:
since $\frac{\partial}{\partial x_i}=I\frac{\partial}{\partial q_i}-\frac{\partial}{\partial \overline{q_i}}I=\frac{\partial}{\partial q_i}I-I\frac{\partial}{\partial \overline{q_i}}$,
\[
\eta(\frac{\partial}{\partial x_i})=\frac{\overline{q_i}I-Iq_i+2q_iI}{2}=t_iI+z_iJ-y_iK.
\]
\[
X_{1,i}=\frac{\partial}{\partial x_i}-\frac{(t_iI+z_iJ-y_iK)}{3}T,
\]
Similarly, plug in $\frac{\partial}{\partial y_i}=J\frac{\partial}{\partial q_i}-\frac{\partial}{\partial \overline{q_i}}J=\frac{\partial}{\partial q_i}J-J\frac{\partial}{\partial \overline{q_i}}$ and $\frac{\partial}{\partial z_i}=K\frac{\partial}{\partial q_i}-\frac{\partial}{\partial \overline{q_i}}K=\frac{\partial}{\partial q_i}K-K\frac{\partial}{\partial \overline{q_i}}$,
\[
\eta(\frac{\partial}{\partial y_i})=\frac{\overline{q_i}J-Jq_i+2q_iJ}{2}=-z_iI+t_iJ+x_iK,
\]
and
\[
\eta(\frac{\partial}{\partial z_i})=\frac{\overline{q_i}K-Kq_i+2q_iK}{2}=y_iI-x_iJ+t_iK.
\]
Therefore we obtain
\[
X_{2,i}=\frac{\partial}{\partial y_i}-\frac{(-z_iI+t_iJ+x_iK)}{3}T, \ X_{3,i}=\frac{\partial}{\partial z_i}-\frac{(y_iI-x_iJ+t_iK)}{3}T.
\]
The sub-Laplacian is then given by
\[
L=\sum_{i=1}^{n+1}X_{0,i}^2+X_{1,i}^2+X_{2,i}^2+X_{3,i}^2,
\]
that is
\[
L=\sum_{i=1}^{n+1}\left(\frac{\partial^2}{\partial t_i^2}+\frac{\partial^2}{\partial x_i^2}+\frac{\partial^2}{\partial y_i^2}+\frac{\partial^2}{\partial z_i^2}\right)-T\overline{T}.
\]
It can be written  in quaternionic coordinates as follows:
\begin{equation}\label{subL}
L= 2\sum_{i=1}^{n+1}\left(\frac{\partial^2 }{\partial q_i\partial \overline{q_i}}+ \frac{\partial^2 }{\partial \overline{q_i}\partial q_i}\right)-(T_1^2+T_2^2+T_3^2).
\end{equation}

 It is hence seen that $L$ is essentially self-adjoint on $C^\infty(\bS^{4n+3})$ with respect to the volume measure and related to the Laplace-Beltrami operator $\Delta$ of the standard Riemannian structure on $\bS^{4n+3}$ by the formula:
\[
L=\Delta -T_1^2-T_2^2-T_3^2,
\]
where $\Delta=\sum_{i=1}^{n+1}\left(\frac{\partial^2}{\partial t_i^2}+\frac{\partial^2}{\partial x_i^2}+\frac{\partial^2}{\partial y_i^2}+\frac{\partial^2}{\partial z_i^2}\right)$ is the Laplacian associated to the Riemannian structure of $\bS^{4n+3}$.

Moreover, we can observe, a fact which will be important for us, that $L$ and $T_i$ commute, that is, on smooth functions $T_i L=LT_i$.

To study $L$ we now introduce a new set of coordinates that reflect the symmetries of the quaternionic contact structure.  Keeping in mind the  submersion $\mathbb{S}^{4n+3} \rightarrow \mathbb{HP}^n$, we let $(w_1,\cdots, w_n,\theta_1,\theta_2,\theta_3)$ be local coordinates for $\bS^{4n+3}$, where $(w_1,\cdots,w_n)$ are the local inhomogeneous coordinates for $\mathbb{HP}^n$ given by $w_j=q_jq_{n+1}^{-1}$
, and $\theta_1,\theta_2,\theta_3$ are local coordinates on the $\mathbf{SU}(2)$ fiber. More explicitly, these coordinates are given by
\begin{align}\label{cylinder}
(w_1,\cdots,w_n,\theta_1,\theta_2,\theta_3 )\longrightarrow \left(\frac{w_ne^{ I\theta_1 +J\theta_2 +K\theta_3} }{\sqrt{1+\rho^2}},\cdots,\frac{w_ne^{ I\theta_1 +J\theta_2 +K\theta_3}}{\sqrt{1+\rho^2}},\frac{e^{ I\theta_1 +J\theta_2 +K\theta_3}}{\sqrt{1+\rho^2}} \right),
\end{align}
where $\rho=\sqrt{\sum_{j=1}^{n}|w_j|^2}$, $\theta_i \in \R/2\pi\mathbb{Z}$, and $w \in \mathbb{HP}^n$.\
Therefore by considering the diffeomorphism 
\[
(w_1,\cdots,w_n,\theta_1,\theta_2,\theta_3,\kappa )\longrightarrow \left(\frac{w_1\kappa e^{ I\theta_1 +J\theta_2 +K\theta_3} }{\sqrt{1+\rho^2}},\cdots,\frac{\kappa w_ne^{ I\theta_1 +J\theta_2 +K\theta_3}}{\sqrt{1+\rho^2}},\frac{\kappa e^{ I\theta_1 +J\theta_2 +K\theta_3}}{\sqrt{1+\rho^2}} \right)
\]
and restrict to 
\[
\kappa=1,\quad \frac{\partial}{\partial\kappa}=0,
\]
we have  that on $\bS^{4n+3}$, for all $1\leq j\leq n$:
%
\begin{eqnarray*}
\frac{\partial}{\partial q_j} &=& q_{n+1}^{-1}\frac{\partial}{\partial w_j}=\sqrt{1+\rho^2}e^{-(I\theta_1+J\theta_2+K\theta_3)}\frac{\partial}{\partial w_j}\\
\frac{\partial}{\partial \overline{q_j}} &=& \frac{\partial}{\partial \overline{w_j}}\overline{q_{n+1}}^{-1}=\sqrt{1+\rho^2}\frac{\partial}{\partial \overline{w_j}}e^{I\theta_1+J\theta_2+K\theta_3}.
\end{eqnarray*}
Moreover, obviously we have that
\[
\frac{\partial}{\partial q_{n+1}} = -q_{n+1}^{-1}\sum_{i=1}^nw_i\frac{\partial}{\partial w_i}+q_{n+1}^{-1}(\sum_{i=1}^{n+1}q_i\frac{\partial}{\partial q_i}),
\]
and when restricted to $\bS^{4n+2}$, we have $\sum_{i=1}^{n+1}q_i\frac{\partial}{\partial q_i}=\frac{1}{2}(IT_1+JT_2+KT_3)$.
Thus
\begin{eqnarray*}
\frac{\partial}{\partial q_{n+1}} &=& -q_{n+1}^{-1}\sum_{i=1}^nw_i\frac{\partial}{\partial w_i}-\frac{1}{2}q_{n+1}^{-1}(IT_1+ JT_2+KT_3)\\
\frac{\partial}{\partial \overline{q_{n+1}}} &=&- \sum_{i=1}^n\frac{\partial}{\partial \overline{w_i}}\overline{w_i}\ (\overline{q_{n+1}})^{-1}+\frac{1}{2}\left(I(\overline{q_{n+1}})^{-1}T_1+ J(\overline{q_{n+1}})^{-1}T_2+K(\overline{q_{n+1}})^{-1}T_3\right)
\end{eqnarray*}

Plug into \eqref{subL} we obtain that
\begin{eqnarray*}
L&=& 2(1+\rho^2)\bigg[\sum_{i=1}^{n+1}\left(\frac{\partial^2}{\partial w_j\partial \overline{w_j}}+\frac{\partial^2}{\partial  \overline{w_j}\partial w_j}\right)
+\left(\mathcal{R}\overline{\mathcal{R}}+\overline{\mathcal{R}} \mathcal{R}\right)\\
&+&(\overline{\mathcal{R}}I-I\mathcal{R})T_1+ (\overline{\mathcal{R}}J-J\mathcal{R})T_2+(\overline{\mathcal{R}}K-K\mathcal{R})T_3)\bigg]
+\rho^2(T_1^2+T_2^2+T_3^2).
\end{eqnarray*}
where $\mathcal{R}=w_i\frac{\partial}{\partial w_i}$.

It is obvious that the subelliptic heat kernel is cylindric symmetric, i.e. it only depends on two coordinates $(r,\eta)$ where $r$ and $\eta$ are the Riemannian distance from the north pole on $\mathbb{HP}^n$ and $\mathbf{SU}(2)$ respectively. Therefore we just need to write the cylindrical part of $L$, denote as $\tilde{L}$. We define it rigorously as follows:
\begin{definition} Let us denote by $\psi$ the map from $\bS^{4n+3} $ to $  \mathbb{R}_{\ge 0} \times \R/2\pi\mathbb{Z} $ such that 
\[
\psi  \left(\frac{w_1e^{ I\theta_1 +J\theta_2 +K\theta_3} }{\sqrt{1+\rho^2}},\cdots,\frac{w_ne^{ I\theta_1 +J\theta_2 +K\theta_3}}{\sqrt{1+\rho^2}},\frac{e^{ I\theta_1 +J\theta_2 +K\theta_3}}{\sqrt{1+\rho^2}} \right)=\left(\rho, \eta \right),
\]
where $\rho=\sqrt{\sum_{j=1}^{n}|w_j|^2}$ and $\eta$ is the Riemannian distance from the identity on $\mathbf{SU}(2)$. We denote by $\mathcal{D}$ the space of smooth and compactly supported functions on $  \mathbb{R}_{\ge 0} \times \R/2\pi\mathbb{Z} $. Then the cylindrical part of $L$ is defined by $\tilde{L}:\mathcal{D}\to C^\infty( \bS^{4n+3} )$ such that for every  $f \in \mathcal{D}$, we have
\[
L(f \circ \psi)=(\tilde{L} f) \circ \psi.
\]
\end{definition}

We now compute the sub-Laplacian in cylindric coordinates.
\begin{proposition}\label{ger}
The cylindric part of the sub-Laplacian on $\bS^{4n+3}$ is given in the coordinates $(r,\eta)$ by
\begin{equation}\label{radial-L}
\tilde{L}=\frac{\partial^2}{\partial r^2}+((4n-1)\cot r-3\tan r)\frac{\partial}{\partial r}+\tan^2r \left(\frac{\partial^2}{\partial \eta^2}+2\cot \eta\frac{\partial}{\partial \eta}\right).
\end{equation}
\end{proposition}
\begin{proof}
Noticing that
\[
\frac{\partial}{\partial w_i}( f\circ \psi)=\frac{1}{2\rho} \left( \frac{\partial f }{\partial \rho} \circ \psi \right)\overline{w_i},\quad
\frac{\partial}{\partial \overline{w_i}}( f\circ \psi)=\frac{1}{2\rho} \left( \frac{\partial f }{\partial \rho} \circ \psi \right) w_i,
\]
and $\frac{\partial}{\partial w_i}w_i=\frac{1}{2}\left(\frac{\partial}{\partial t_i^w}-I\frac{\partial}{\partial x_i^w}-J\frac{\partial}{\partial y_i^w}-K\frac{\partial}{\partial z_i^w}\right)(t_i^w+Ix_i^w+Jy_i^w+Kz_i^w)=2$, 
we can compute that 
\[
\sum_{i=1}^{n+1}\frac{\partial^2}{\partial w_j\partial \overline{w_j}}(f \circ \psi)=\sum_{i=1}^{n+1}\frac{\partial^2}{\partial \overline{w_j}\partial w_j}(f \circ \psi)=\left(\frac{1}{4}\frac{\partial^2f}{\partial \rho^2}+\frac{4n-1}{4\rho}\frac{\partial f}{\partial \rho} \right) \circ \psi.
\]
Moreover,
\[
\mathcal{R}(f \circ \psi)=\overline{\mathcal{R}}( f \circ \psi) =\frac{\rho}{2}\frac{\partial f }{\partial \rho} \circ \psi,
\]
and we have that
\[
(\overline{\mathcal{R}}I-I\mathcal{R})T_1( f \circ \psi) =0.
\]
By the same reason we have that  $(\overline{\mathcal{R}}J-J\mathcal{R})T_2(f\circ \psi) $ and $(\overline{\mathcal{R}}K-K\mathcal{R})T_3(f \circ \psi)$ vanish as well. This yields,
\[
\mathcal{R}\overline{\mathcal{R}}(f \circ \psi)=\overline{\mathcal{R}}\mathcal{R}f= \left( \frac{3\rho}{4}\frac{\partial f}{\partial \rho}+\frac{\rho^2}{4}\frac{\partial^2f}{\partial \rho^2}\right) \circ \psi.
\]
Hence we obtain after simplifications
\[
L(f \circ \psi)=\left( (1+\rho^2)^2\frac{\partial^2 f}{\partial \rho^2} +\left(\frac{(4n-1)(1+\rho^2)}{\rho}+3(1+\rho^2)\right)\frac{\partial f}{\partial \rho} \right) \circ \psi
+\rho^2(T_1^2+T_2^2+T_3^2)(f\circ \psi).
\]

Notice that $T_1,T_2,T_3$  is the canonical basis on $\mathbf{SU}(2)$, hence $T_1^2+T_2^2+T_3^2$ is the Laplacian on $\mathbf{SU}(2)$, and 
\[
(T_1^2+T_2^2+T_3^2)(f \circ \psi) =\left(\frac{\partial^2 f}{\partial \eta^2}+2\cot \eta\frac{\partial f}{\partial \eta}\right)\circ \psi.
\]
We then obtain the desired expression by observing that $\rho=\tan r$.
\end{proof}

As a consequence of the previous result, it is an easy exercise to check that the Riemannian measure of $\bS^{4n+3} $ is given in cylindric coordinates $(r,\eta)$ by
\[
\frac{8\pi^{2n+1}}{\Gamma(2n)}(\sin r)^{4n-1}(\cos r)^3(\sin\eta)^2drd\eta.
\]

\begin{remark}
From \eqref{radial-L} one can easily observe  that the cylindric  sub-Laplacian on $\BS$ can thus  be obtained by lifting the radial sub-Laplacian $\mathbb{HP}^n$, more precisely
\[
\tilde{L}= \tilde{\Delta}_{\mathbb{HP}^n}+\tan ^2 r \tilde{\Delta}_{\mathbf{SU}(2)},
\]
where $\tilde{\Delta}_{\mathbb{HP}^n}$ is the radial part of the Laplacian on $\mathbb{HP}^n$ and $\tilde{\Delta}_{\mathbf{SU}(2)}$ is the radial part of the Laplacian on $\mathbf{SU}(2)$.
\end{remark}

\subsection{Spectral decomposition of the heat kernel}\label{spectral}

In this section, we derive the spectral decomposition of the subelliptic heat kernel of the heat semigroup $P_t=e^{tL}$ issued from the north pole. Notice that due to the cylindric symmetry, the heat kernel that we denote $p_t(r,\eta)$ will only depend on the coordinates $(r,\eta)$. We now prove the following Minakshisundaram-Pleijel expansion theorem:
\begin{proposition}
For $t>0$, $r\in[0,\frac{\pi}{2})$, $ \eta\in[0,\pi]$, the subelliptic kernel is given by
\begin{equation}\label{pt}
p_t(r,\eta)=\sum_{m=0}^{+\infty}\sum_{k=0}^{\infty}\alpha_{k,m}e^{-4[k(k+2n+m+1)+nm]t}\frac{\sin (m+1)\eta}{\sin \eta}(\cos r)^mP_k^{2n-1,m+1}(\cos 2r)
\end{equation}
where $\alpha_{k,m}=\frac{\Gamma(2n)}{2\pi^{2n+2}}(2k+m+2n+1)(m+1){k+m+2n\choose 2n-1}$ and 
\[
P_k^{2n-1,m+1}(x)=\frac{(-1)^k}{2^kk!(1-x)^{2n-1}(1+x)^{m+1}}\frac{d^k}{dx^k}\left((1-x)^{2n-1+k}(1+x)^{m+1+k} \right).
\] 
is a Jacobi polynomial.
\end{proposition} 
\begin{proof}
The idea is to   expand the subelliptic kernel in spherical harmonics as follows, 
\[
p_t(r,\eta)=\sum_{m=0}^{+\infty}\frac{\sin (m+1)\eta}{\sin \eta}\phi_m(t,r)
\]
where $\frac{\sin (m+1)\eta}{\sin \eta}$ is the eigenfunction of $\tilde{\Delta}_{SU(2)}=\frac{\partial^2}{\partial \eta^2}+2\cot \eta\frac{\partial}{\partial \eta}$ which is associated to the eigenvalue $-m(m+2)$. To determine $\phi_m$, we use  $\frac{\partial p_t}{\partial t}=\tilde{L}p_t$ and find that
\[
\frac{\partial\phi_m}{\partial t}=\frac{\partial^2\phi_m}{\partial r^2}+\left((4n-1)\cot r-3\tan r \right)\frac{\partial\phi_m}{\partial r}-m(m+2)\tan^2r\phi_m.
\]
Let $\phi_m(t,r)=e^{-4nmt}(\cos r)^m\varphi_m(t,r)$, then $\varphi_m(t,r)$ satisfies the equation
\[
\frac{\partial\varphi _m}{\partial t}=\frac{\partial^2\varphi_m}{\partial r^2}+[(4n-1)\cot r-(2m+3)\tan r]\frac{\partial\varphi_m}{\partial r}.
\]
We now change the variable and denote by $\varphi_m(t,r)=g_m(t,\cos 2r)$, then we have that $g_m(t,x)$ satisfies the equation
\[
\frac{\partial g_m}{\partial t}=4(1-x^2)\frac{\partial^2 g_m}{\partial x^2}+4[(m+2-2n)-(2n+m+2)x]\frac{\partial g_m}{\partial x}.
\]
We denote $\Psi_m=(1-x^2)\frac{\partial^2 }{\partial x^2}+[(m+2-2n)-(2n+m+2)x]\frac{\partial}{\partial x}$, and find that
\[
\frac{\partial g_m}{\partial t}=4\Psi_m(g_m).
\]
The equation $$\Psi_m(g_m)+k(k+2n+m+1)g_m=0$$ is a Jacobi differential equation for all $k\geq 0$.  We denote the eigenvector of $\Psi_m$ corresponding to the eigenvalue $-k(k+2n+m+1)$ by $P_k^{2n-1,m+1}(x)$, then it is known that 
\[
P_k^{2n-1,m+1}(x)=\frac{(-1)^k}{2^kk!(1-x)^{2n-1}(1+x)^{m+1}}\frac{d^k}{dx^k}\left((1-x)^{2n-1+k}(1+x)^{m+1+k} \right).
\] 
At the end we can therefore write the spectral decomposition as
\[
p_t(r,\eta)=\sum_{m=0}^{+\infty}\sum_{k=0}^{\infty}\alpha_{k,m}e^{-4[k(k+2n+m+1)+nm]t}\frac{\sin (m+1)\eta}{\sin \eta}(\cos r)^mP_k^{2n-1,m+1}(\cos 2r)
\]
where $\alpha_{k,m}$ are  determined by considering the initial condition.

Note that $(P_k^{2n-1,m+1}(x)(1+x)^{(m+1)/2})_{k\geq0}$ is an orthogonal basis of the Hilbert space $L^2([-1,1],(1-x)^{2n-1}dx)$, more precisely
\[
\int_{-1}^1 P_k^{2n-1,m+1}(x)P_l^{2n-1,m+1}(x)(1-x)^{2n-1}(1+x)^{m+1}dx=\frac{2^{2n+m+1}}{2k+m+2n+1}\frac{\Gamma(k+2n)\Gamma(k+m+2)}{\Gamma(k+1)\Gamma(k+2n+m+1)}\delta_{kl}.
\]
For a smooth function $f(r, \theta)$, we can write
\[
f(r, \eta)=\sum_{m=0}^{+\infty}\sum_{k=0}^{+\infty} b_{k,m}\frac{\sin (m+1)\eta}{\sin \eta}P_k^{2n-1,m+1}(\cos 2r)\cdot(\cos r)^{m}
\]
where the $ b_{k,m}$'s are constants. We obtain then
\[
f(0,0)=\sum_{m=0}^{+\infty}\sum_{k=0}^{+\infty} b_{k,m}(m+1)P_k^{2n-1,m+1}(1).
\]
and we observe that $P_k^{2n-1,m+1}(1)={2n-1+k\choose k}$.
The  measure $d\mu$ is given in cylindric coordinates by
\[
d\mu_r=\frac{8\pi^{2n+1}}{\Gamma(2n)}(\sin r)^{4n-1}(\cos r)^3(\sin\eta)^2drd\eta
\]
Moreover, since
\begin{eqnarray*}
& &\int_{0}^\pi\int_0^\frac{\pi}{2} p_t(r, \eta){f(-r, -\eta)}d\mu_r \\
&=&\frac{4\pi^{2n+2}}{\Gamma(2n)}\sum_{m=0}^{+\infty}\sum_{k=0}^{+\infty}\alpha_{k,m}b_{k,m}e^{-\lambda_{k,m}t}
\left(\int_0^{\frac{\pi}{2}}(\cos r)^{2m+3}|P_k^{2n-1,m+1}|^2(\sin r)^{4n-1}dr\right)\\
&=&\frac{2\pi^{2n+2}}{\Gamma(2n)} \sum_{m=0}^{+\infty}\sum_{k=0}^{+\infty}\frac{\alpha_{k,m}b_{k,m}e^{-\lambda_{m,k}t}}{2k+m+2n+1}\frac{\Gamma(k+2n)\Gamma(k+m+2)}{\Gamma(k+1)\Gamma(k+2n+m+1)}
\end{eqnarray*}
where $\lambda_{k,m}=4k(k+2n+m+1)+nm$, we obtain that
\[
\lim_{t\rightarrow 0}\int_{0}^\pi\int_0^\frac{\pi}{2} p_tfd\mu_r= f(0,0)
\]
as soon as $\alpha_{k,m}=\frac{\Gamma(2n)}{2\pi^{2n+2}}(2k+m+2n+1)(m+1){k+m+2n\choose 2n-1}$.
\end{proof}

Comparing this expansion with a result we obtained in  \cite{CRS}, we obtain a very nice formula relating $p_t$ to  the heat kernel of the Hopf fibration
 \[
  \mathbb{S}^1 \to \bS^{4n+1} \to  \mathbb{CP}^{2n}.
 \]
More precisely, we proved that   the subelliptic kernel $p_t^{CR}(r, \theta)$ of the above fibration writes:
\[
p_t^{CR}(r, \theta)=\frac{\Gamma(2n)}{2\pi^{2n+1}}\sum_{m=-\infty}^{+\infty}\sum_{k=0}^{+\infty} (2k+|m|+2n){k+|m|+2n-1\choose 2n-1}e^{-\lambda_{k,m}t+im \theta}(\cos r)^{|m|}P_k^{2n-1,|m|}(\cos 2r),
\]
where $\lambda_{k,m}=4k(k+|m|+2n)+4|m|n$, and $r$, $\theta$ are the Riemannian distance on $\mathbb{CP}^{2n}$, $\bS^1$ respectively. By comparing the kernels on $\BS$ and $\bS^{4n+1}$, we easily obtain that
\begin{proposition}
Let $p_t^{CR}(r, \theta)$ and $p_t(r, \theta)$ denote the subelliptic heat kernels on the CR sphere $\bS^{4n+1}$ and the quaternionic sphere $\BS$ respectively , then for $r\in[0,\frac{\pi}{2})$, $ \theta\in[0,\pi]$,
\begin{equation}\label{pt-qt}
-\frac{e^{4nt}}{2\pi\sin \theta\cos r}\frac{\partial}{\partial \theta} p_t^{CR}(r, \theta)=p_t(r,\theta).
\end{equation}
\end{proposition}

\subsection{Integral representation of the subelliptic heat kernel}

Our goal in this section is to provide an alternative representation of $p_t(r,\eta)$ which shall later be useful in the purpose of studying small-times asymptotics.

From \eqref{radial-L} we can easily see that $\tilde{L}=\tilde{\Delta}_{\bS^{4n+3}}-\tilde{\Delta}_{\mathbf{SU}(2)}$, and $\tilde{L}$ commutes with $\tilde{\Delta}_{\mathbf{SU}(2)}$. Thus, we  have that
\[
e^{t\tilde{L}}=e^{-t\tilde{\Delta}_{\mathbf{SU}(2)}}e^{t\tilde{\Delta}_{\bS^{4n+3}}}.
\]

We denote by $q_t$ the heat kernel of the heat semigroup $e^{t\tilde{\Delta}_{\bS^{4n+3}}}$, then  the subelliptic heat kernel $p_t(r,\eta)$ can be obtained by applying the heat semigroup $e^{-t\tilde{\Delta}_{\mathbf{SU}(2)}}$ on $q_t$, i.e.
\[
p_t(r,\eta)=(e^{-t\tilde{\Delta}_{\mathbf{SU}(2)}}q_t)(r,\eta).
\]
We denote by $h_t(\eta)$ the heat kernel of $\tilde{\Delta}_{\mathbf{SU}(2)}$, that is, $h_t$ is the fundamental solution of 
\[
\frac{\partial }{\partial t}h_t(\eta)=\left(\frac{\partial^2}{\partial\eta^2}+2\cot\eta\frac{\partial}{\partial\eta}\right)h_t(\eta).
\]
Consider $g_t(\eta)=h_t(i\eta)$, then $g_t(\eta)$ satisfies 
\[
\frac{\partial }{\partial t}g_t(\eta)=-\left(\frac{\partial^2}{\partial\eta^2}+2\coth\eta\frac{\partial}{\partial\eta}\right)g_t(\eta).
\]
If we denote $\tilde{\Delta}_{\mathbf{SL}(2)}=\frac{\partial^2}{\partial \eta ^2}+2 \coth \eta \frac{\partial}{\partial \eta} $, then it is not hard to see that
\begin{equation}\label{pt-qts}
p_t(r,\eta)=(e^{t\tilde{\Delta}_{\mathbf{SL}(2)}}q_t)(r,-i\eta).
\end{equation}

As a conclusion, an integral representation of $p_t$,  can be obtained from an explicit expression of the heat semigroup $e^{t\tilde{\Delta}_{\mathbf{SL}(2)}}$. 
\begin{lemma}
Let $\tilde{\Delta}_{\mathbf{SL}(2)}=\frac{\partial^2}{\partial \eta ^2}+2 \coth \eta \frac{\partial}{\partial \eta} $. For every $f:\mathbb{R}_{\ge 0} \to \mathbb{R}$ in the domain of $\tilde{\Delta}_{\mathbf{SL}(2)}$, we have:
\begin{equation}\label{semigroup}
(e^{t\tilde{\Delta}_{\mathbf{SL}(2)}} f)(\eta)=\frac{e^{-t}}{\sqrt{\pi t}} \int_0^{+\infty} \frac{ \sinh r \sinh \left(  \frac{\eta r}{2t}\right) }{\sinh \eta} e^{-\frac{r^2+\eta^2}{4t}} f( r ) dr, \quad t \ge 0, \eta \ge 0.
\end{equation}
\end{lemma}

\begin{proof}
We observe that:
\[
\tilde{\Delta}_{\mathbf{SL}(2)} f=\frac{1}{h} (\tilde{\Delta}_{\R^3}-1)(hf),
\]
where
\[
\tilde{\Delta}_{\R^3} =\frac{\partial^2}{\partial \eta ^2}+\frac{2}{  \eta} \frac{\partial}{\partial \eta} , \quad h(\eta)=\frac{\sinh \eta} {\eta}.
\]
As a consequence, we have
\[
(e^{t\tilde{\Delta}_{\mathbf{SL}(2)}} f)(\eta)=\frac{e^{-t}}{h(\eta)}e^{t \tilde{\Delta}_{\R^3}} (hf) (\eta).
\]
We are thus let with the computation of $e^{t \tilde{\Delta}_{\R^3}}$. The operator $\tilde{\Delta}_{\R^3}$ is the radial part of the Laplacian $\Delta_{\mathbb{R}^3}$, thus for $x \in \mathbb{R}^3$,
\[
e^{t \Delta_{\mathbb{R}^3}}( f \circ r)(x)= (e^{t\tilde{\Delta}_{\R^3}}f )(r(x)),
\]
where $r(x)=\| x \|$. Since,
\[
e^{t \Delta_{\mathbb{R}^3}}( f \circ r)(x)=\frac{1}{(4\pi t )^{3/2}} \int_{\mathbb{R}^3} e^{-\frac{\| y-x\|^2}{4t} } f(\|y \|) dy,
\]
a routine computation in spherical coordinates shows that
\[
e^{t\tilde{\Delta}_{\R^3}}f (\eta)=\frac{1}{\sqrt{\pi t}} \int_0^{+\infty} \frac{r}{\eta} \sinh \left( \frac{\eta r}{2t}\right) e^{-\frac{r^2+\eta^2}{4t}} f ( r ) dr.
\]
The conclusion easily follows.
\end{proof}

Now we can immediately deduce the integral representation of the subelliptic heat kernel on $\bS^{4n+3}$: 
\begin{proposition}\label{prop-pt}
For $t>0$, $r\in[0,\pi/2)$, $ \eta\in[0,\pi]$, 
\begin{equation}\label{pt-int} 
p_t(r, \eta)=\frac{e^{-t}}{\sqrt{\pi t}} \int_0^{+\infty} \frac{ \sinh y \sin \left(  \frac{\eta y}{2t}\right) }{\sin \eta} e^{-\frac{y^2-\eta^2}{4t}} q_t( \cos r\cosh y ) dy.
\end{equation}
\end{proposition}

\begin{proof}
Because of \eqref{pt-qts}, by plugging in $q_t$ to \eqref{semigroup}, we immediately obtain the desired integral representation.
\end{proof}

For later use, we record here that the heat kernel $q_t$  on $\bS^{4n+3}$ with respect to its Riemannian structure has been well studied, here we list two useful representations of it: 
\begin{itemize}
\item[(1)] The spectral decomposition of $q_t$ is given by
\begin{equation}
q_t{(\cos\delta)}=\frac{\Gamma(2n+1)}{2\pi^{2n+2}}\sum_{m=0}^{+\infty}(m+2n+1)e^{-m(m+4n+2)t}C_m^{2n+1}(\cos \delta),
\end{equation}
where $\delta$ is the Riemannian distance from the north pole and
\[
C_m^{2n+1}(x)=\frac{(-1)^m}{2^m}\frac{\Gamma(m+4n+2)\Gamma(2n+3/2)}{\Gamma(4n+2)\Gamma(m+1)\Gamma(2n+m+3/2)}\frac{1}{(1-x^2)^{2n+1/2}}\frac{d^m}{dx^m}(1-x^2)^{2n+m+1/2}
\]
is a Gegenbauer polynomial.  
\item[(2)]Another expression of $q_t (\cos \delta)$ which is useful for the computation of small-time asymptotics is 
\begin{equation}\label{heat_kernel_odd}
q_t (\cos \delta)= e^{(2n+1)^2t} \left( -\frac{1}{2\pi \sin \delta} \frac{\partial}{\partial \delta} \right)^{2n+1} V,
\end{equation}
where $V(t,\delta)=\frac{1}{\sqrt{4\pi t}} \sum_{k \in \mathbb{Z}} e^{-\frac{(\delta-2k\pi)^2}{4t} }$. 
\end{itemize}

\begin{remark}
By using a result we proved  in \cite{CRS},  we can give an alternative proof of Proposition \ref{prop-pt}. From the result in \cite{CRS} we know that the subelliptic heat kernel of the fibration
\[
  \mathbb{S}^1 \to \bS^{4n+1} \to  \mathbb{CP}^{2n}
  \]
   writes
\[
p_t^{CR}(r,\eta)=\frac{1}{\sqrt{4\pi t}}\int_{-\infty}^{+\infty}e^{-\frac{y^2}{4t} }q_t^{R}(\cos r\cosh (y-i\eta))dy,
\] 
where $q^R_t$ is the Riemannian heat kernel  on $\bS^{4n+1}$. Plug it in to \eqref{pt-qt} we obtain that
\[
p_t(r,\eta)=-\frac{e^{4nt}}{2\pi\sin \eta\cos r}\frac{1}{\sqrt{4\pi t}}\int_{-\infty}^{+\infty}e^{-\frac{y^2}{4t} }\frac{\partial}{\partial \eta} q_t^{R}(\cos r\cosh (y-i\eta))dy
\]
Moreover, since it is well-known (see \cite{F}) that
\[
\frac{e^{-{(4n+1)t}}}{2\pi}\frac{d}{dx}q_t^R(x)=q_t(x),
\]
we obtain that
\begin{eqnarray*}
p_t(r,\eta)= \frac{e^{-t}}{\sin \eta}\frac{1}{\sqrt{\pi t}}\int_0^{+\infty}e^{-\frac{y^2-\eta^2}{4t} }\sin{\left(\frac{y\eta}{2t} \right)}q_t(\cos r\cosh y)\sinh (y)dy,
\end{eqnarray*}
which agrees with  \eqref{pt-int}.
\end{remark}

%

\subsection{The Green function of the conformal sub-Laplacian}
The first consequence of  \eqref{pt-int} is  the Green function of the conformal sub-Laplacian $-L+4n(n+1)$.
\begin{theorem}
For all $r\in[0,\pi/2), \eta\in[0,\pi]$, the Green function of the conformal sub-Laplacian $-L+4n(n+1)$ is given by
\[
G(r,\eta)=\frac{\Gamma(n)\Gamma(n+1)}{8\pi^{2n+2}(1-2\cos r\cos\eta+\cos^2 r)^{n+1}}.
\]
\end{theorem}
\begin{proof}
From \cite{CRS}, we already  know that
\[
\int_0^{+\infty}p_t^{CR}(r, \eta)e^{-4n^2t}dt=\frac{\Gamma\left(n\right)^2}{8\pi^{2n+1}(1-2\cos r\cos \theta+\cos^2r)^{n}} 
\]
Plug in \eqref{pt-qt}, we immediately have that
\[
\int_0^{+\infty}p_t(r, \eta)e^{-4(n^2+n)t}dt=-\frac{1}{2\pi\cos r\sin\eta}\frac{\partial}{\partial\eta}\left(\frac{\Gamma(n)^2}{8\pi^{2n+2}(1-2\cos r\cos\eta+\cos^2 r)^{n}} \right),
\]
thus 
\[
G(r,\eta)=\frac{\Gamma(n)\Gamma(n+1)}{8\pi^{2n+2}(1-2\cos r\cos\eta+\cos^2 r)^{n+1}}.
\]
\end{proof}

\subsection{Small time asymptotics}

Another advantage of the integral representation  \eqref{pt-int} is to deduce the small time asymptotics of the subelliptic heat kernel. The following small time asymptotics for the Riemannian heat kernel is well-known
\begin{align}\label{eq9}
q_t(\cos\delta)=\frac{1}{(4\pi t)^{2n+\frac{3}{2}}}\left(\frac{\delta}{\sin\delta}\right)^{2n+1}e^{-\frac{\delta^2}{4t}}\left(1+\left((2n+1)^2-\frac{2n(2n+1)(\sin\delta-\delta\cos\delta)}{\delta^2\sin\delta}\right)t+O(t^2)\right),
\end{align}
where $\delta \in [0,\pi)$  is the Riemannian distance from the north pole. By applying \eqref{pt-int}, one can then deduce the  small time asymptotics of the subelliptic heat kernel.
\begin{prop}
When $t\to 0$, 
\[
p_t(0,0)=\frac{1}{(4\pi t)^{2n+3}}(A_n+B_nt+O(t^2)),
\]
where $A_n=4\pi\int_0^{+\infty}\frac{y^{2n+2}}{(\sinh y)^{2n}}dy$ and $B_n=4\pi\int_0^{+\infty}\frac{y^{(2n+2)}}{(\sinh y)^{2n}}\left(4n^2+4n-\frac{2n(2n+1)(\sinh y-y\cosh y)}{y^2\sinh y}\right)dy$.
\end{prop}
\begin{proof}
We know  that
\begin{eqnarray*}
p_t(0,0)= \frac{e^{-t}}{\sqrt{\pi t}}\int_{0}^{+\infty}\sinh y\frac{y}{2t}e^{-\frac{y^2}{4t}}q_t(\cosh y)dy.
\end{eqnarray*}
Plug in \eqref{eq9}, we have that
\[
p_t(0,0)= \frac{4\pi e^{-t}}{(4\pi t)^{2n+3}}\int_{-\infty}^{+\infty}\frac{y^{2n+2}}{(\sinh y)^{2n}}\left(1+\left((2n+1)^2-\frac{2n(2n+1)(\sinh y-y\cosh y)}{y^2\sinh y}\right)t\right)dy.
\]
Thus we can obtain the small time asymptotic as follows:
\[
p_t(0,0)=\frac{1}{(4\pi t)^{2n+3}}(A_n+B_nt+O(t^2)),
\]
where $A_n=4\pi\int_0^{+\infty}\frac{y^{2n+2}}{(\sinh y)^{2n}}dy$ and $B_n=4\pi\int_0^{+\infty}\frac{y^{(2n+2)}}{(\sinh y)^{2n}}\left(4n^2+4n-\frac{2n(2n+1)(\sinh y-y\cosh y)}{y^2\sinh y}\right)dy$.
\end{proof}

The small time behavior of the subelliptic heat kernel on the vertical cut-locus, namely the points $(0,\eta)$ that can be achieved by flowing along vertical vector fields is quite different. A short-cut to deduce it is by differentiating the small time estimate of  $p_t^{CR}(0,\eta)$.
\begin{prop}\label{asy-cut}
For $ \eta\in(0,\pi)$, $t\rightarrow 0$,
\[
p_t(0, \eta)=\frac{1}{4\pi\sin \eta{2^{6n}t^{4n+1}(2n-1)!}}\left( (\pi - \eta) \eta^{2n-1}e^{-\frac{2\pi \eta- \eta^2}{4t}}   \right)(1+O(t)).
\]
\end{prop}
\begin{proof}
Since   
\[
-\frac{e^{4nt}}{2\pi\sin \eta\cos r}\frac{\partial}{\partial \eta} p_t^{CR}(r, \eta)=p_t(r,\eta),
\]
we just need to plug in the small time asymptotic of $p_t^{CR}$ on the cut locus:
\[
p_t^{CR}(0, \eta)=\frac{ \eta^{2n-1}}{2^{6n}t^{4n}(2n-1)!}e^{-\frac{2\pi \eta- \eta^2}{4t}}(1+O(t)),
\]
we then have that
\[
p_t(0, \eta)=\frac{1}{4\pi\sin \eta{2^{6n}t^{4n+1}(2n-1)!}}\left( (\pi - \eta) \eta^{2n-1}e^{-\frac{2\pi \eta- \eta^2}{4t}}   \right)(1+O(t)).
\]
\end{proof}

We now deduce the small time behavior of the kernel on the horizontal space of $\BS$. i.e. $(r,0)$, $r\not=0$.
\begin{prop}\label{propr0}
For $r\in(0,\frac{\pi}{2})$,  we have
\[
p_t(r,0)\sim_{t\to 0} \frac{1}{(4\pi t)^{2n+\frac{3}{2}}}\left(\frac{r}{\sin r}\right)^{2n+1}e^{-\frac{r^2}{4t}}\left(\frac{1}{1-r\cot r}\right)^{\frac{3}{2}}.
\]
\end{prop}
\begin{proof}
By Proposition \ref{prop-pt},
\[
p_t(r,0)=\frac{e^{-t}}{\sqrt{4\pi t}}\int_{-\infty}^{+\infty}(\sinh y)\frac{y}{2t}e^{-\frac{y^2}{4t}}q_t(\cos r\cosh y)dy,
\]
and by plugging in (\ref{eq9}), we obtain that 
\[
p_t(r,0)\sim_{t\to 0}\frac{1}{(4\pi t)^{2n+2}}(J_1(t)+J_2(t))
\]
where 
\[
J_1(t)=\int_{\cosh y\leq\frac{1}{\cos r}}e^{-\frac{y^2+(\arccos (\cos r\cosh y))^2}{4t}}\left(\frac{ y}{2t}\right)\sinh y
\left(\frac{\arccos (\cos r\cosh y)}{\sqrt{1-\cos^2r\cosh^2y}} \right)^{2n+1}dy
\]
and
\[
J_2(t)=\int_{\cosh y\geq\frac{1}{\cos r}}e^{-\frac{y^2-(\cosh^{-1} (\cos r\cosh y))^2}{4t}}\left(\frac{ y}{2t}\right)\sinh y
\left(\frac{\cosh^{-1} (\cos r\cosh y)}{\sqrt{\cos^2r\cosh^2y-1}} \right)^{2n+1}dy.
\]
The idea is to analyze $J_1(t)$ and $J_2(t)$ by Laplace method. \\
First, notice that in $\left[-\cosh^{-1}(\frac{1}{\cos r}),\cosh^{-1}(\frac{1}{\cos r})\right]$, $f(y)=y^2+(\arccos (\cos r\cosh y))^2$ has a unique minimum at $y=0$, where
\[
f^{''}(0)=2(1-r\cot r).
\] 
Hence by Laplace method, we can easily obtain that
\[
J_1(t)\sim_{t\rightarrow 0}e^{-\frac{r^2}{4t}}\left(\frac{r}{\sin r}\right)^{2n+1}\left(\frac{1}{1-r\cot r}\right)^{\frac{3}{2}}\sqrt{4\pi t}
\]
On the other hand, on $\left(-\infty,-\cosh^{-1}(\frac{1}{\cos r})\right)\cup\left(\cosh^{-1}(\frac{1}{\cos r}),+\infty\right)$, the function 
\[
g(y)=y^2-(\cosh^{-1} (\cos r\cosh y))^2
\]
has no minimum, which implies that $J_2(t)$ is negligible with respect to $J_1(t)$ in small $t$. Hence the conclusion.
\end{proof}

For the case  $(r,\eta)$,  with $r\not=0$,  the Laplace method no longer works, we need to use the steepest descent method.
 
\begin{prop}\label{asym}
Let  $r\in(0,\frac{\pi}{2})$, $\eta\in[0,\pi]$. 
Then when $t\to 0$,
\begin{equation}\label{sta}
p_t(r,\eta)\sim_{t\to 0}-\frac{1}{(4\pi t)^{2n+\frac{3}{2}}} \frac{\sin\varphi(r,\eta)}{\sin\eta\sin r}\frac{(\arccos u(r,\eta))^{2n+1}}{\sqrt{1-\frac{u(r,\eta)\arccos u(r,\eta)}{\sqrt{1-u^2(r,\eta)}}}}\frac{e^{-\frac{(\varphi(r,\eta)+\eta)^2\tan^2 r}{4t\sin^2(\varphi(r,\eta))}}}{(1-u(r,\eta)^2)^{n}},
\end{equation}
where $u(r,\eta)=\cos r\cos\varphi(r,\eta)$ and $\varphi(r,\eta)$ is the unique solution in $[0,\pi]$ to the equation
\begin{equation}\label{varphi}
\varphi(r,\eta)+\eta=\cos r\sin\varphi(r,\eta)\frac{\arccos(\cos \varphi(r,\eta)\cos r)}{\sqrt{1-\cos^2r\cos^2\varphi(r,\eta)}},
\end{equation}

\end{prop}
\begin{proof}
Since 
\[
p_t(r, \eta)=\frac{e^{-t}}{\sqrt{4\pi t}} \int_{-\infty}^{+\infty} \frac{ \sinh y }{2i\sin \eta}\left(e^{-\frac{(y-i\eta)^2}{4t}}-e^{-\frac{(y+i\eta)^2 }{4t}}\right)q_t( \cos r\cosh y ) dy,
\]
by plugging in \eqref{eq9}, we obtain that
\[
p_t(r,\eta)\sim_{t\rightarrow 0}\frac{1}{(4\pi t)^{2n+2}} (A-B),
\]
where
\[
A=\int_{-\infty}^{+\infty} \frac{ \sinh y }{2i\sin \eta}\left(e^{-\frac{(y-i\eta)^2+(\arccos (\cos r\cosh y))^2}{4t}}\right) \left(\frac{\arccos (\cos r\cosh y)}{\sqrt{1-\cos^2r\cosh^2y}} \right)^{2n+1}dy
\]
and
\[
B=\int_{-\infty}^{+\infty} \frac{ \sinh y }{2i\sin \eta}\left(e^{-\frac{(y+i\eta)^2+(\arccos (\cos r\cosh y))^2 }{4t}}\right) \left(\frac{\arccos (\cos r\cosh y)}{\sqrt{1-\cos^2r\cosh^2y}} \right)^{2n+1}dy.
\]
We first deduce the small time asymptotic of $B$: By applying the steepest descent method, we can constraint the integral on the strip  $|\mathbf{Re}(y)|<\cosh^{-1}\left( \frac{1}{\cos r}\right)$ where, due to the result in \cite{CRS} (Lemma 3.9), 
\[
f(y)=(y+i\eta)^2+(\arccos (\cos r \cosh y))^2
\]
has a critical point at $i\varphi(r,\eta)$. $\varphi(r,\eta)$ is the unique solution in $[0,\pi]$ to the equation
\eqref{varphi} and 
\[
f^{''}(i\varphi(r,\eta))=\frac{2\sin^2 r}{1-u(r,\eta)^2}\left(1-\frac{u(r,\eta)\arccos u(r,\eta)}{\sqrt{1-u^2(r,\eta)}} \right)
\]
is positive, where $u(r,\eta)=\cos r\cos\varphi(r,\eta)$. 
Thus in small time, $B$ has the following estimates:
\begin{equation}\label{B}
B\sim_{t\to 0}\frac{\sqrt{4\pi t}\sin\varphi(r,\eta)}{2\sin\eta\sin r}\frac{(\arccos u(r,\eta))^{2n+1}}{\sqrt{1-\frac{u(r,\eta)\arccos u(r,\eta)}{\sqrt{1-u^2(r,\eta)}}}}\frac{e^{-\frac{(\varphi(r,\eta)+\eta)^2\tan^2 r}{4t\sin^2(\varphi(r,\eta))}}}{(1-u(r,\eta)^2)^{n}}.
\end{equation} 
To estimate $A$, we denote 
\[
g(y)=(y-i\eta)^2+(\arccos (\cos r \cosh y))^2.
\]
Simple computations show that $g(y)$ has a critical point at $-i\varphi(r,\eta)$ where $\varphi(r,\eta)$ is as described in \eqref{varphi}. Thus
\begin{equation}\label{A}
A\sim_{t\to 0}-\frac{\sqrt{4\pi t}\sin\varphi(r,\eta)}{2\sin\eta\sin r}\frac{(\arccos u(r,\eta))^{2n+1}}{\sqrt{1-\frac{u(r,\eta)\arccos u(r,\eta)}{\sqrt{1-u^2(r,\eta)}}}}\frac{e^{-\frac{(\varphi(r,\eta)+\eta)^2\tan^2 r}{4t\sin^2(\varphi(r,\eta))}}}{(1-u(r,\eta)^2)^{n}}.
\end{equation} 
Therefore we obtain \eqref{sta} by plugging in \eqref{A} and \eqref{B}.
\end{proof}

\begin{remark}
In proposition \ref{asym}, if we let $\eta=0$,  \eqref{varphi} has a unique solution at $\varphi=0$ and
\[
\lim_{\eta\to0} \frac{\varphi(r,\eta)}{\eta}=-\frac{1}{1-r\cot r}.
\]
Thus \eqref{sta} gives that
\[
p_t(r,0)\sim_{t\to 0} \frac{1}{(4\pi t)^{2n+\frac{3}{2}}}\left(\frac{r}{\sin r}\right)^{2n+1}e^{-\frac{r^2}{4t}}\left(\frac{1}{1-r\cot r}\right)^{\frac{3}{2}},
\]
which, as it should be,  agrees with the result in Proposition \ref{propr0}.
\end{remark}

\begin{remark}
By symmetry, the sub-Riemannian distance from the north pole to any point on $\bS^{4n+3}$ only depends on $r$ and $\eta$. If we denote it by $d(r,\eta)$, then from the previous propositions,
\item[(1)]
For $\eta\in [0,\pi]$,
\[
d^2(0,\eta)=2\pi \eta-\eta^2
\]
\item[(2)]
For $\eta\in [0,\pi]$, $r\in\left(0,\frac{\pi}{2}\right)$,
\[
d^2(r,\eta)=\frac{(\varphi(r,\eta)+\theta)^2\tan^2 r}{\sin^2(\varphi(r,\eta))}
\]
In particular,  the sub-Riemannian diameter of $\bS^{4n+3}$ is $\pi$.
\end{remark}

\section{The subelliptic heat kernel on $\mathbb{CP}^{2n+1}$}

We now turn to the second main object of our study.
In this section we study the subelliptic heat kernel of the Riemannian submersion  $\mathbb{CP}^{2n+1} \to\mathbb{HP}^n$ which is induced from the quaternionic Hopf fibration. 
\subsection{Projected Hopf fibration}

Besides the action of  $\mathbf{SU}(2)$ on $\mathbb{S}^{4n+3}$ that induces the quaternionic Hopf fibration which was studied in the previous sections, we can also consider the action of $\mathbb{S}^1$ on $\mathbb{S}^{4n+3}$ that induces the classical Hopf fibration:
\begin{diagram}
\mathbb{S}^1 & \rTo & \mathbb{S}^{4n+3}  & \rTo &  \mathbb{CP}^{2n+1}.
\end{diagram}
We can then see $\mathbb{S}^1$ as a subgroup of $\mathbf{SU}(2)$ and deduce a a fibration
\begin{diagram}
\mathbf{SU}(2)  / \mathbb{S}^1 =  \mathbb{CP}^1 & \rTo & \mathbb{CP}^{2n+1} & \rTo & \mathbb{HP}^n
\end{diagram}
that makes the following diagram commutative
\begin{diagram}\label{diag}
  & & \mathbb{S}^1 & & \\
  & \ldTo & \dTo & & \\
 \mathbf{SU}(2) & \rTo &\mathbb{S}^{4n+3} & \rTo & \mathbb{HP}^n \\
 \dTo & &\dTo & \ruTo & \\
 \mathbb{CP}^1 & \rTo & \mathbb{CP}^{2n+1} & & 
\end{diagram}
We consider the sub-Laplacian ${\mathcal{L}}$ on  $\mathbb{CP}^{2n+1}$ which is the lift of the Laplace-Beltrami operator of  $\mathbb{HP}^n$. From the above diagram, ${\mathcal{L}}$ is also the projection of the sub-Laplacian $L$ of $\mathbb{S}^{4n+3} $ on $\mathbb{CP}^{2n+1}$. As we have seen before,  the radial part of $L$ is 
\[
\tilde{L}=\frac{\partial^2}{\partial r^2}+((4n-1)\cot r-3\tan r)\frac{\partial}{\partial r}+\tan^2r (\frac{\partial^2}{\partial \eta^2}+2\cot \eta\frac{\partial}{\partial \eta})
\]
where $r$ is the Riemannian distance on $\mathbb{HP}^n$ and $\eta$ the Riemannian distance on $ \mathbf{SU}(2) $. The operator
\[
\tilde{\Delta}_{SU(2)}=\frac{\partial^2}{\partial \eta^2}+2\cot \eta\frac{\partial}{\partial \eta}
\]
is the radial part of the Laplace-Beltrami operator on $ \mathbf{SU}(2) $. As it has been proved in Baudoin-Bonnefont (see \cite{BB}), by using the fibration,
\begin{diagram}
\mathbb{S}^1 & \rTo & \mathbf{SU}(2)  & \rTo &  \mathbb{CP}^{1}
\end{diagram}
we can write
\[
\frac{\partial^2}{\partial \eta^2}+2\cot \eta\frac{\partial}{\partial \eta}=\frac{\partial^2}{\partial \phi^2}+2\cot 2\phi\frac{\partial}{\partial \phi}+(1+\tan^2\phi )\frac{\partial^2}{\partial \theta^2}
\]
where $\phi$ is the Riemannian distance on $\mathbb{CP}^{1}$ and $\frac{\partial}{\partial \theta}$ the generator of the action of $\mathbb{S}^1$ on $\mathbf{SU}(2)$. We deduce that $\tilde{L}$ can be written as
\[
\frac{\partial^2}{\partial r^2}+((4n-1)\cot r-3\tan r)\frac{\partial}{\partial r}+\tan^2r\left( \frac{\partial^2}{\partial \phi^2}+2\cot 2\phi\frac{\partial}{\partial \phi}+(1+\tan^2\phi )\frac{\partial^2}{\partial \theta^2}\right)
\]
Therefore the radial part of the sub-Laplacian $\mathcal{L}$ on  $\mathbb{CP}^{2n+1}$ is given by 
\begin{equation}\label{L-CP}
\tilde{\mathcal{L}}=\frac{\partial^2}{\partial r^2}+((4n-1)\cot r-3\tan r)\frac{\partial}{\partial r}+\tan^2r\left( \frac{\partial^2}{\partial \phi^2}+2\cot 2\phi\frac{\partial}{\partial \phi}\right).
\end{equation}

We can also see , and this will be later used, that the Riemannian measure reads, up to a normalization, $(\sin r)^{4n-1} (\cos r)^3 \sin 2\phi dr d\phi$.
\subsection{Spectral decomposition of the subelliptic heat kernel on $\mathbb{CP}^{2n+1}$}
We now study the subelliptic heat kernel $h_t(r,\phi)$ associated to $\mathcal{L}$. Similarly as in Proposition \ref{spectral}, we get the spectral decomposition of $h_t$.
\begin{proposition}
For $t>0$, $r\in[0,\frac{\pi}{2})$, $ \phi\in[0,\pi]$, the subelliptic kernel is given by
\[
h_t(r,\phi)=\sum_{m=0}^{+\infty}\sum_{k=0}^{\infty}\sigma_{k,m}e^{-[4k(k+2n+2m+1)+2nm]t}(\cos r)^{2m}P_m^{0,0}(\cos2\phi)P_k^{2n-1,2m+1}(\cos 2r),
\]
where $\sigma_{k,m}=\frac{\Gamma(2n)}{4\pi^{2n+2}}(2k+2m+2n+1)(2m+1){k+2m+2n\choose 2n-1}$.
\end{proposition}
\begin{proof}
From \eqref{L-CP} we notice that $\frac{\partial^2}{\partial \phi^2}+2\cot 2\phi\frac{\partial}{\partial \phi}$ is the Laplacian on $\mathbb{CP}^1$.  It is known that the eigenfunction associated to the eigenvalue $-4m(m+1)$ is given by $P_m^{0,0}(\cos2\phi)$ where
\[
P_m^{0,0}(x)=\frac{(-1)^m}{2^m m!}\frac{d^m}{dx^m}(1-x^2)^{m},
\]
thus we can write the spectral decomposition of $h_t(r,\phi)$ as follows:
\[
h_t(r,\phi)=\sum_{m=0}^{+\infty}P_m^{0,0}(\cos2\phi)\phi_m(t,r).
\]
Plug it into the heat equation $\frac{\partial h_t}{\partial t}=\mathcal{L}h_t$, we obtain that
\[
\frac{\partial\phi_m}{\partial t}=\frac{\partial^2\phi_m}{\partial r^2}+\left((4n-1)\cot r-3\tan r \right)\frac{\partial\phi_m}{\partial r}-4m(m+1)\tan^2r\phi_m.
\]
Furthermore we consider the decomposition $\phi_m(t,r)=e^{-8nmt}(\cos r)^{2m}\varphi_m(t,r)$ where $\varphi_m(t,r)$ satisfies 
\[
\frac{\partial\varphi _m}{\partial t}=4\frac{\partial^2\varphi_m}{\partial r^2}+[(4n-1)\cot r-(4m+3)\tan r]\frac{\partial\varphi_m}{\partial r},
\]
and let $g_m(t,\cos 2r)=\varphi_m(t,r)$, then $g_m(t,x)$ is such that
\[
\frac{\partial g_m}{\partial t}=4(1-x^2)\frac{\partial^2 g_m}{\partial x^2}+4[(2m+2-2n)-(2n+2m+2)x]\frac{\partial g_m}{\partial x}.
\]
We denote 
\[
\Psi_m=(1-x^2)\frac{\partial^2 }{\partial x^2}+4[(2m+2-2n)-(2n+2m+2)x]\frac{\partial}{\partial x},
\]
then
\[
\frac{\partial g_m}{\partial t}=4\Psi_m(g_m).
\]
Moreover, for any $k\geq 0$, $$\Psi_m(g_m)+k(k+2n+2m+1)g_m=0$$  is the Jacobi differential equation whose eigenvector corresponding to the eigenvalue $-k(k+2n+2m+1)$  is given by 
\[
P_k^{2n-1,2m+1}(x)=\frac{(-1)^k}{2^kk!(1-x)^{2n-1}(1+x)^{2m+1}}\frac{d^k}{dx^k}\left((1-x)^{2n-1+k}(1+x)^{2m+1+k} \right).
\] 
Thus the spectral decomposition of $h_t(r,\phi)$ has the following form:
\[
h_t(r,\phi)=\sum_{m=0}^{+\infty}\sum_{k=0}^{\infty}\sigma_{k,m}e^{-[4k(k+2n+2m+1)+2nm]t}(\cos r)^{2m}P_m^{0,0}(\cos2\phi)P_k^{2n-1,2m+1}(\cos 2r),
\]
and by plugging in the initial condition, one can determine that  $\sigma_{k,m}=\frac{\Gamma(2n)}{4\pi^{2n+2}}(2k+2m+2n+1)(2m+1){k+2m+2n\choose 2n-1}$.
\end{proof}
We now compare the spectral decompositions of $h_t$ and $p_t$ and obtain the relations between $h_t$ and $p_t$ as follows.
\begin{proposition}
For $t>0$, $r\in[0,\frac{\pi}{2})$, $ \phi\in[0,\pi]$,
\begin{equation}\label{htpt}
h_t(r,\cos2\phi)=\frac{1}{2\pi}\int_0^\pi p_t(r,\cos\phi\cos\theta)d\theta
\end{equation}
where $p_t$ is the subelliptic heat kernel of the sub-Laplacian $L$ on $\BS$.
\end{proposition}
\begin{proof} 
Let us use the expression of $p_t$ in \eqref{pt} and recall (see \cite{F}) the fact that
\[
P_m^{0,0}(\cos2\phi)=\frac{1}{\pi}\int_0^\pi U_{2m+1}(\cos\phi\cos\theta)d\theta
\]
where $U_{2m+1}$ is the Chebyshev polynomial
\[
U_{2m+1}(\cos\phi)=\frac{\sin (2m+1)\phi}{\sin\phi}.
\]
Since $\cos\eta=\cos\phi\cos\theta$, we have that
\[
h_t(r,\phi)=\sum_{m=0}^{+\infty}\sum_{k=0}^{\infty}\sigma_{k,m}e^{-[4k(k+2n+2m+1)+2nm]t}(\cos r)^{2m}\left(\frac{1}{\pi}\int_0^\pi U_{2m+1}(\cos\phi\cos\theta)d\theta\right) P_k^{2n-1,2m+1}(\cos 2r),
\]
which implies the desired relation between the subelliptic heat kernels. The even terms on the right hand side of \eqref{htpt} vanish due to the fact that
\[
p_t(r+\pi,\eta)=p_t(r,\eta).
\]
\end{proof}
Immediately from \eqref{htpt}, we can  deduce the integral representation of the subelliptic heat kernel on $\mathbb{CP}^{2n+1}$:
\begin{theorem}
For $t>0$, $r\in[0,\pi/2)$, $ \phi\in[0,\pi]$, 
\begin{equation}\label{kernel-cpn}
h_t(r,\cos2\phi)= \frac{e^{-t}}{2\pi\sqrt{\pi t}}\int_0^\pi \int_0^{+\infty} \frac{ \sinh y \sin \left(  \frac{ y\arccos(\cos \phi\cos\theta)}{2t}\right) }{\sqrt{1-\cos^2\phi\cos^2\theta}} e^{-\frac{y^2-\arccos(\cos \phi\cos\theta)^2}{4t}} q_t( \cos r\cosh y ) dyd\theta
\end{equation}
where $q_t$ is the Riemannian heat kernel on $S^{4n+3}$.
\end{theorem}

\subsection{Small time asymptotics of the subelliptic kernel on $\cp^{2n+1}$}

With \eqref{kernel-cpn} we now study the small time behavior of the subelliptic kernel $h_t$. 

\begin{proposition}
When $t\to0$:
\[
h_t(0,0)\sim_{t\to0}\frac{1}{(2n-1)2^{4n+4}\pi^{2n}t^{4n+2}}.
\] 
\end{proposition}
\begin{proof}
From \eqref{htpt} we know that
\[
h_t(0, 0)=\frac{e^{-t}}{2\pi\sqrt{4\pi t}}\int_0^{\pi} \int_{-\infty}^{+\infty} \frac{ \sinh y \sin \left(  \frac{\theta y}{2t}\right) }{\sin \theta} e^{-\frac{y^2-\theta^2}{4t}} q_t^S( \cosh y ) dy,
\]
By applying the residue theorem we can easily deduce that
\[
h_t(0, 0)\sim\frac{1}{2\pi(2n-1)!2^{6n+2}t^{4n+1}}\int_0^\pi\frac{\theta^{2n-1}}{\sin\theta}e^{\frac{\theta^2-2\pi \theta}{4t}}d\theta.
\]
Notice that for $\epsilon>0$ small enough,
\[
\int_0^\pi\frac{\theta^{2n-1}}{\sin\theta}e^{\frac{\theta^2-2\pi \theta}{4t}}d\theta\sim
\int_0^\epsilon\theta^{2n-2}e^{\frac{-2\pi \theta}{4t}}d\theta,
\] where the last term has the following estimates:
\[
\int_0^\epsilon\theta^{2n-2}e^{\frac{-2\pi \theta}{4t}}d\theta\sim\Gamma(2n-1)\left(\frac{2t}{\pi}\right)^{2n-1}.
\]
Therefore we have the conclusion.
\end{proof}

\begin{proposition}
On the vertical cut-locus, i.e. for any $(0,\phi)$, $\phi\in(0,\pi)$, 
\[
h_t(0,\phi)\sim_{t\to0} \frac{1}{2\pi(2n-1)!2^{6n+2}t^{4n+\frac{1}{2}}}e^{\frac{\phi^2-2\pi\phi}{4t}}\frac{\phi^{2n-1}}{\sqrt{\sin 2\phi}} \sqrt{\frac{2\pi }{\pi-\phi}}.
\]
\end{proposition}
\begin{proof}
By \eqref{kernel-cpn} we have that
\[
h_t(0,\cos2\phi)= \frac{e^{-t}}{2\pi\sqrt{\pi t}}\int_0^\pi \int_0^{+\infty} \frac{ \sinh y \sin \left(  \frac{ y\arccos(\cos \phi\cos\theta)}{2t}\right) }{\sqrt{1-\cos^2\phi\cos^2\theta}} e^{-\frac{y^2-\arccos(\cos \phi\cos\theta)^2}{4t}} q_t( \cosh y ) dyd\theta.
\]
Apply the same technique as in Proposition \ref{asy-cut}, we then obtain that
\[
h_t(0,\phi)\sim_{t\to0} \frac{1}{2\pi}\int_0^\pi \frac{u^{2n-1}}{(2n-1)!2^{6n+2}t^{4n+1}\sin u}e^{\frac{u^2-2\pi u}{4t}}d\theta,
\] 
where $u=\arccos(\cos\phi\cos\theta)$. Thus by changing variables we have that  for all $\phi\in(0,\pi)$,
\[
h_t(0,\phi)\sim_{t\to0} \frac{1}{2\pi(2n-1)!2^{6n+2}t^{4n+1}}\int_\phi^{\phi+\pi} \frac{u^{2n-1}}{\sqrt{\cos^2\phi-\cos^2 u}}e^{\frac{u^2-2\pi u}{4t}}du.
\]
Since
\begin{eqnarray*}
\int_\phi^{\phi+\pi} \frac{u^{2n-1}}{\sqrt{\cos^2\phi-\cos^2 u}}e^{\frac{u^2-2\pi u}{4t}}du
&=&
\sqrt{2}\int_0^{\pi}\frac{(\phi+x)^{2n-1}}{\sqrt{\cos 2\phi-\cos 2(\phi+x)}}e^{\frac{(\phi+x)^2-2\pi(\phi+x)}{4t}}dx\\
&\sim& e^{\frac{\phi^2-2\pi\phi}{4t}}\int_0^\epsilon\frac{\phi^{2n-1}}{\sqrt{\sin 2\phi}\sqrt{x}}e^{-\frac{(\pi-\phi)x}{2t}}dx,
\end{eqnarray*}
By Laplace method we obtain 
\[
\int_0^\epsilon\frac{1}{\sqrt{x}}e^{-\frac{(\pi-\phi)x}{2t}}dx\sim \sqrt{\frac{2\pi t}{\pi-\phi}},
\]
thus
\[
h_t(0,\phi)\sim_{t\to0} \frac{1}{2\pi(2n-1)!2^{6n+2}t^{4n+\frac{1}{2}}}e^{\frac{\phi^2-2\pi\phi}{4t}}\frac{\phi^{2n-1}}{\sqrt{\sin 2\phi}} \sqrt{\frac{2\pi }{\pi-\phi}}.
\]
\end{proof}
\begin{remark}
For the points outside of the cut-locus, the small time asymptotics  follows the result by Ben Arous in \cite{Ben1}: for $t>0$, $r,\in(0,\frac{\pi}{2})$, $\phi\in(0,\pi)$,
\[
h_t(r,\phi)\sim\frac{C(r,\phi)}{t^{2n+\frac{3}{2}}}e^{-\frac{d(r,\phi)^2}{4t}},
\]
where $d$ is the sub-Riemannian distance. Being more explicit about the distance $d$ seems to be quite difficult.
\end{remark}

\clearpage
\end{document}